\documentclass{amsart}
\usepackage{amsfonts}

\newtheorem{theorem}{Theorem}
\theoremstyle{plain}

\newtheorem{definition}{Definition}

\newtheorem{lemma}{Lemma}

\numberwithin{equation}{section}
\input{tcilatex}

\begin{document}
\title[On Hermite Hadamard-type inequalities]{On Hermite Hadamard-type
inequalities for strongly $\varphi $-convex functions }
\author{Mehmet Zeki SARIKAYA}
\address{Department of Mathematics, \ Faculty of Science and Arts, D\"{u}zce
University, D\"{u}zce-TURKEY}
\email{sarikayamz@gmail.com}

\begin{abstract}
In this paper, we introduce the notion of strongly $\varphi $-convex
functions with respect to $c>0$ and present some properties and
representation of such functions. We obtain a characterization of inner
product spaces involving the notion of strongly $\varphi $-convex functions.
Finaly, a version of Hermite Hadamard-type inequalities for strongly $%
\varphi $-convex functions are established.
\end{abstract}

\maketitle

\section{Introduction}

The inequalities discovered by C. Hermite and J. Hadamard for convex
functions are very important in the literature (see, e.g.,\cite{dragomir1},%
\cite[p.137]{pecaric}). These inequalities state that if $f:I\rightarrow 
\mathbb{R}$ is a convex function on the interval $I$ of real numbers and $%
a,b\in I$ with $a<b$, then 
\begin{equation}
f\left( \frac{a+b}{2}\right) \leq \frac{1}{b-a}\int_{a}^{b}f(x)dx\leq \frac{%
f\left( a\right) +f\left( b\right) }{2}.  \label{E1}
\end{equation}%
The inequality (\ref{E1}) has evoked the interest of many mathematicians.
Especially in the last three decades numerous generalizations, variants and
extensions of this inequality have been obtained, to mention a few, see (%
\cite{bakula}-\cite{set2}) and the references cited therein.

Let us consider a function $\varphi :[a,b]\rightarrow \lbrack a,b]$ where $%
[a,b]\subset \mathbb{R}$. Youness have defined the $\varphi $-convex
functions in \cite{youness}:

\begin{definition}
A function $f:[a,b]\rightarrow \mathbb{R}$ is said to be $\varphi $- convex
on $[a,b]$ if for every two points $x\in \lbrack a,b],y\in \lbrack a,b]$ and 
$t\in \lbrack 0,1]$ the following inequality holds:%
\begin{equation*}
f(t\varphi (x)+(1-t)\varphi (y))\leq tf(\varphi (x))+(1-t)f(\varphi (y)).
\end{equation*}
\end{definition}

In \cite{cristescu1}, Cristescu proved the followig results for the $\varphi 
$-convex functions

\begin{lemma}
\label{l} For $f:[a,b]\rightarrow \mathbb{R}$, the following statements are
equivalent:

(i) $f$ is $\varphi $-convex functions on $[a,b]$,

(ii) for every $x,y\in \lbrack a,b]$, the mapping $g:[0,1]\rightarrow 
\mathbb{R},\ g(t)=f(t\varphi (x)+(1-t)\varphi (y))$ is classically convex on 
$[0,1].$
\end{lemma}

Obviously, if function $\varphi $ is the identity, then the classical
convexity is obtained from the previous definition. Many properties of the $%
\varphi $-convex functions can be found, for instance, in \cite{cristescu}, 
\cite{cristescu1},\cite{youness}.

Recall also that a function $f:I\rightarrow \mathbb{R}$ is called strongly
convex with modulus $c>0,$ if%
\begin{equation*}
f\left( tx+\left( 1-t\right) y\right) \leq tf\left( x\right) +\left(
1-t\right) f\left( y\right) -ct(1-t)(x-y)^{2}
\end{equation*}%
for all $x,y\in I$ and $t\in (0,1).$ Strongly convex functions have been
introduced by Polyak in \cite{polyak}\ and they play an important role in
optimization theory and mathematical economics. Various properties and
applicatins of them can be found in the literature see (\cite{polyak}-\cite%
{angu}) and the references cited therein.

In this paper, we introduce the notion of strongly $\varphi $-convex
functions defined in normed spaces and present some properties of them. In
particular, we obtain a representation of strongly $\varphi $-convex
functions in inner product spaces and, using the methods of \cite{nik} and 
\cite{angu}, we give a characterization of inner product spaces, among
normed spaces, that involves the notion of strongly $\varphi $-convex
function. Finally, a version of Hermite--Hadamard-type inequalities for
strongly $\varphi $-convex functions is presented. This result generalizes
the Hermite--Hadamard-type inequalities obtained in \cite{mer} for strongly
convex functions, and for $c=0$, coincides with the classical
Hermite--Hadamard inequalities, as well as the corresponding
Hermite--Hadamard-type inequalities for $\varphi $-convex functions in \cite%
{cristescu1}.

\section{Main Results{}}

In what follows $(X,\left\Vert .\right\Vert )$ denotes a real normed space, $%
D$ stands for a convex subset of $X$, $\varphi :D\rightarrow D$ is a given
function and $c$ is a positive constant. We say that a function $%
f:D\rightarrow \mathbb{R}$ is strongly $\varphi $-convex with modulus $c$ if

\begin{equation}
f(t\varphi (x)+(1-t)\varphi (y))\leq tf(\varphi (x))+(1-t)f(\varphi
(y))-ct(1-t)\left\Vert \varphi (x)-\varphi (y)\right\Vert ^{2}  \label{1}
\end{equation}%
for all $x,y\in D$ and $t\in \lbrack 0,1]$. We say that $f$ is strongly $%
\varphi $-midconvex with modulus $c$ if \ (\ref{1}) is assumed only for $t=%
\frac{1}{2}$, that is%
\begin{equation*}
f\left( \frac{\varphi (x)+\varphi (y)}{2}\right) \leq \frac{f(\varphi
(x))+(f(\varphi (y))}{2}-\frac{c}{4}\left\Vert \varphi (x)-\varphi
(y)\right\Vert ^{2},\ \text{for }x,y\in D.
\end{equation*}%
The notion of $\varphi $-convex function corresponds to the case $c=0$. We
start with the following lemma which give some relationships between
strongly $\varphi $-convex functions and $\varphi $-convex functions in the
case where $X$ is a real inner product space (that is, the norm $\left\Vert
.\right\Vert $ is induced by an inner product: $\left\Vert .\right\Vert
:=<x|x>$).

\begin{lemma}
\label{z} Let $(X,\left\Vert .\right\Vert )$ be a real inner product space, $%
D$ be a convex subset of $X$ and $c$ be a positive constant and $\varphi
:D\rightarrow D$.

i) A function $f:D\rightarrow \mathbb{R}$ is strongly $\varphi $-convex with
modulus $c$ if and only if the function $g=f-c\left\Vert .\right\Vert ^{2}$
is $\varphi $-convex.

ii) A function $f:D\rightarrow \mathbb{R}$ is strongly $\varphi $-midconvex
with modulus $c$ if and only if the function $g=f-c\left\Vert .\right\Vert
^{2}$ is $\varphi $-midconvex.
\end{lemma}

\begin{proof}
i) Assume that $f$ is strongly $\varphi $-convex with modulus $c.$ Using
properties of the inner product, we obtain%
\begin{eqnarray*}
&&g(t\varphi (x)+(1-t)\varphi (y)) \\
&& \\
&=&f(t\varphi (x)+(1-t)\varphi (y))-c\left\Vert t\varphi (x)+(1-t)\varphi
(y)\right\Vert ^{2} \\
&& \\
&\leq &tf(\varphi (x))+(1-t)f(\varphi (y))-ct(1-t)\left\Vert \varphi
(x)-\varphi (y)\right\Vert ^{2}-c\left\Vert t\varphi (x)+(1-t)\varphi
(y)\right\Vert ^{2} \\
&& \\
&\leq &tf(\varphi (x))+(1-t)f(\varphi (y))-c\left( t(1-t)\left[ \left\Vert
\varphi (x)\right\Vert ^{2}-2<\varphi (x)|\varphi (y)>+\left\Vert \varphi
(y)\right\Vert ^{2}\right] \right. \\
&& \\
&&\left. -\left[ t^{2}\left\Vert \varphi (x)\right\Vert ^{2}+2t(1-t)<\varphi
(x)|\varphi (y)>+(1-t)\left\Vert \varphi (y)\right\Vert ^{2}\right] \right)
\\
&& \\
&=&tf(\varphi (x))+(1-t)f(\varphi (y))-ct\left\Vert \varphi (x)\right\Vert
^{2}-c(1-t)\left\Vert \varphi (y)\right\Vert ^{2} \\
&& \\
&=&tg(\varphi (x))+(1-t)g(\varphi (y))
\end{eqnarray*}%
which gives that $g$ is $\varphi $-convex function.

Conversely, if $g$ is $\varphi $-convex function, then we get%
\begin{eqnarray*}
f(t\varphi (x)+(1-t)\varphi (y)) &=&g(t\varphi (x)+(1-t)\varphi
(y))+c\left\Vert t\varphi (x)+(1-t)\varphi (y)\right\Vert ^{2} \\
&& \\
&\leq &tg(\varphi (x))+(1-t)g(\varphi (y))+c\left\Vert t\varphi
(x)+(1-t)\varphi (y)\right\Vert ^{2} \\
&& \\
&=&t\left[ g(\varphi (x))+c\left\Vert \varphi (x)\right\Vert ^{2}\right]
+(1-t)\left[ g(\varphi (y))+c\left\Vert \varphi (y)\right\Vert ^{2}\right] \\
&& \\
&&-ct(1-t)\left[ \left\Vert \varphi (x)\right\Vert ^{2}-2<\varphi
(x)|\varphi (y)>+\left\Vert \varphi (y)\right\Vert ^{2}\right] \\
&& \\
&=&tf(\varphi (x))+(1-t)f(\varphi (y))-ct(1-t)\left\Vert \varphi (x)-\varphi
(y)\right\Vert ^{2}
\end{eqnarray*}%
which shows that $f$ is strongly $\varphi $-convex with modulus $c.$

ii) Assume now that $f$ is strongly $\varphi $-midconvex with modulus $c.$
Using the parallelogram law, we have%
\begin{eqnarray*}
g\left( \frac{\varphi (x)+\varphi (y)}{2}\right) &=&f\left( \frac{\varphi
(x)+\varphi (y)}{2}\right) -c\left\Vert \frac{\varphi (x)+\varphi (y)}{2}%
\right\Vert ^{2} \\
&& \\
&\leq &\frac{f(\varphi (x))+f(\varphi (y))}{2}-\frac{c}{4}\left\Vert \varphi
(x)-\varphi (y)\right\Vert ^{2}-\frac{c}{4}\left\Vert \varphi (x)+\varphi
(y)\right\Vert ^{2} \\
&& \\
&=&\frac{f(\varphi (x))+f(\varphi (y))}{2}-\frac{c}{4}\left( 2\left\Vert
\varphi (x)\right\Vert ^{2}+2\left\Vert \varphi (y)\right\Vert ^{2}\right) \\
&& \\
&=&\frac{g(\varphi (x))+g(\varphi (y))}{2}
\end{eqnarray*}%
which gives that $g$ is $\varphi $-midconvex function.

Similarly, if $g$ is $\varphi $-midconvex function, then we get%
\begin{eqnarray*}
f\left( \frac{\varphi (x)+\varphi (y)}{2}\right) &=&g\left( \frac{\varphi
(x)+\varphi (y)}{2}\right) +c\left\Vert \frac{\varphi (x)+\varphi (y)}{2}%
\right\Vert ^{2} \\
&& \\
&\leq &\frac{g(\varphi (x))+g(\varphi (y))}{2}+\frac{c}{4}\left\Vert \varphi
(x)+\varphi (y)\right\Vert ^{2} \\
&& \\
&=&\frac{g(\varphi (x))+\left\Vert \varphi (x)\right\Vert ^{2}}{2}+\frac{%
g(\varphi (y))+\left\Vert \varphi (y)\right\Vert ^{2}}{2} \\
&& \\
&&+\frac{c}{4}\left( \left\Vert \varphi (x)+\varphi (y)\right\Vert
^{2}-2\left\Vert \varphi (x)\right\Vert ^{2}-2\left\Vert \varphi
(y)\right\Vert ^{2}\right) \\
&& \\
&=&\frac{f(\varphi (x))+f(\varphi (y))}{2}-\frac{c}{4}\left\Vert \varphi
(x)-\varphi (y)\right\Vert ^{2}.
\end{eqnarray*}%
This completes to proof.
\end{proof}

The following example shows that the assumption that $X$ is an inner product
space is essentials in the above lemma.

\textbf{Example.} Let $X=\mathbb{R}^{2}$. Let us consider a function $%
\varphi :\mathbb{R}^{2}\rightarrow \mathbb{R}^{2},$ defined by $\varphi
(x)=x $ for every $x\in \mathbb{R}^{2}$ and $\left\Vert x\right\Vert =\max
\left\{ \left\vert x_{1}\right\vert ,\left\vert x_{2}\right\vert \right\} $
for $x=(x_{1},x_{2})$. Take $f=\left\Vert .\right\Vert ^{2}.$ Then $%
g=f-\left\Vert .\right\Vert ^{2}$ is $\varphi $-convex being the zero
function. However, $f$ is neither strongly $\varphi $-convex with modulus $1$
nor strongly $\varphi $-midconvex with modulus $1.$ Indeed, for $x=(1,0)$
and $y=(0,1),$ we have%
\begin{equation*}
f\left( \frac{x+y}{2}\right) =\frac{1}{2}\geq \frac{3}{4}=\frac{f(x)+f(y)}{2}%
-\frac{1}{4}\left\Vert x-y\right\Vert ^{2}
\end{equation*}%
which this contradicts (\ref{1}).

The assumption that $X$ is an inner product space in Lemma \ref{z} is
essential. Moreover, it appears that the fact that for every $\varphi $%
-convex function $g:X\rightarrow \mathbb{R}$ the function $f=g+c\left\Vert
.\right\Vert ^{2}$ is strongly $\varphi $-convex characterizes inner product
spaces among normed spaces. Similar characterizations of inner product
spaces by strongly convex and strongly $h$-convex functions are presented in 
\cite{nik} and \cite{angu}.

\begin{theorem}
Let $(X,\left\Vert .\right\Vert )$ be a real normed space, $D$ be a convex
subset of $X$ and $\varphi :D\rightarrow D$. Then the following conditions
are equivalent:

i) $(X,\left\Vert .\right\Vert )$ is a real inner product space;

ii) For every $c>0$, $f:D\rightarrow \mathbb{R}$ defined on a convex subset $%
D$ of $X$, the function $f=g+c\left\Vert .\right\Vert ^{2}$ is strongly $%
\varphi $-convex with modulus $c$;

iii) $\left\Vert .\right\Vert ^{2}:X\rightarrow \mathbb{R}$ is strongly $%
\varphi $-convex with modulus $1$.
\end{theorem}

\begin{proof}
The implication i)$\Rightarrow $ii) follows by Lemma \ref{z}. To see that ii)%
$\Rightarrow $iii) take $g=0$. Clearly, $g$ is $\varphi $-convex function,
whence $f=c\left\Vert .\right\Vert ^{2}$ is strongly $\varphi $-convex with
modulus $c$. Consequently, $\left\Vert .\right\Vert ^{2}$ is strongly $%
\varphi $-convex with modulus $1$. Finaly, to prove iii)$\Rightarrow $i)
observe that by the strongly $\varphi $-convexity of $\left\Vert
.\right\Vert ^{2}$, we obtain%
\begin{equation*}
\left\Vert \frac{\varphi (x)+\varphi (y)}{2}\right\Vert ^{2}\leq \frac{%
\left\Vert \varphi (x)\right\Vert ^{2}}{2}+\frac{\left\Vert \varphi
(y)\right\Vert ^{2}}{2}-\frac{1}{4}\left\Vert \varphi (x)-\varphi
(y)\right\Vert ^{2}
\end{equation*}%
and hence%
\begin{equation}
\left\Vert \varphi (x)+\varphi (y)\right\Vert ^{2}+\left\Vert \varphi
(x)-\varphi (y)\right\Vert ^{2}\leq 2\left\Vert \varphi (x)\right\Vert
^{2}+2\left\Vert \varphi (y)\right\Vert ^{2}  \label{2}
\end{equation}%
for all $x,y\in X.$ Now, putting $u=\varphi (x)+\varphi (y)$ and $v=\varphi
(x)-\varphi (y)$ in (\ref{2}), we have%
\begin{equation}
2\left\Vert u\right\Vert ^{2}+2\left\Vert v\right\Vert ^{2}\leq \left\Vert
u+v\right\Vert ^{2}+\left\Vert u-v\right\Vert ^{2}  \label{3}
\end{equation}%
for all $u,v\in X.$

Conditions (\ref{2}) and (\ref{3}) mean that the norm $\left\Vert
.\right\Vert ^{2}$ satisfies the parallelogram law, which implies, by the
classical Jordan-Von Neumann theorem, that $(X,\left\Vert .\right\Vert )$ is
an inner product space. This completes to proof.
\end{proof}

Now, we give a new Hermite--Hadamard-type inequalities for strongly $\varphi 
$-convex functions with modulus $c$ as follows:

\begin{theorem}
If $f:\left[ a,b\right] \rightarrow \mathbb{R}$ is strongly $\varphi $-
convex with modulus $c>0$ for the continuous function $\varphi :\left[ a,b%
\right] \rightarrow \left[ a,b\right] ,$ then%
\begin{eqnarray}
&&f\left( \frac{\varphi (a)+\varphi (b)}{2}\right) +\frac{c}{12}\left(
\varphi (a)-\varphi (b)\right) ^{2}  \label{4} \\
&\leq &\frac{1}{\varphi (b)-\varphi (a)}\dint\limits_{\varphi (a)}^{\varphi
(b)}f(x)dx  \notag \\
&\leq &\frac{f(\varphi (a))+f(\varphi (b))}{2}-\frac{c}{6}\left( \varphi
(a)-\varphi (b)\right) ^{2}.  \notag
\end{eqnarray}
\end{theorem}

\begin{proof}
From the strongly $\varphi $- convexity of $f$, we have%
\begin{eqnarray*}
f\left( \frac{\varphi (a)+\varphi (b)}{2}\right) &=&f\left( \frac{t\varphi
(a)+(1-t)\varphi (b)}{2}+\frac{(1-t)\varphi (a)+t\varphi (b)}{2}\right) \\
&& \\
&\leq &\frac{1}{2}f\left( t\varphi (a)+(1-t)\varphi (b)\right) +\frac{1}{2}%
f\left( (1-t)\varphi (a)+t\varphi (b)\right) -\frac{c}{4}\left( 1-2t\right)
^{2}\left( \varphi (a)-\varphi (b)\right) ^{2}.
\end{eqnarray*}%
By Lemma \ref{l}, it follows that the previous inequality can be integrated
with respect to $t$ over $[0,1]$, getting%
\begin{eqnarray*}
&&f\left( \frac{\varphi (a)+\varphi (b)}{2}\right) +\frac{c}{12}\left(
\varphi (a)-\varphi (b)\right) ^{2} \\
&& \\
&\leq &\frac{1}{2}\dint\limits_{0}^{1}f\left( t\varphi (a)+(1-t)\varphi
(b)\right) dt+\frac{1}{2}\dint\limits_{0}^{1}f\left( (1-t)\varphi
(a)+t\varphi (b)\right) dt.
\end{eqnarray*}%
In the first integral, we substitute $x=t\varphi (a)+(1-t)\varphi (b)$.
Meanwhile, in the second integral we also use the substitution $%
x=(1-t)\varphi (a)+t\varphi (b)$, we obtain%
\begin{eqnarray*}
&&f\left( \frac{\varphi (a)+\varphi (b)}{2}\right) +\frac{c}{12}\left(
\varphi (a)-\varphi (b)\right) ^{2} \\
&& \\
&\leq &\frac{1}{2\left( \varphi (a)-\varphi (b)\right) }\dint\limits_{%
\varphi (b)}^{\varphi (a)}f\left( x\right) dx+\frac{1}{2\left( \varphi
(b)-\varphi (a)\right) }\dint\limits_{\varphi (a)}^{\varphi (b)}f\left(
x\right) dx \\
&& \\
&=&\frac{1}{\varphi (b)-\varphi (a)}\dint\limits_{\varphi (a)}^{\varphi
(b)}f(x)dx.
\end{eqnarray*}

In order to prove the second inequality, we start from the strongly $\varphi 
$- convexity of $f$ meaning that for every $t\in \lbrack 0,1]$ one has 
\begin{equation*}
f(t\varphi (a)+(1-t)\varphi (b))\leq tf(\varphi (a))+(1-t)f(\varphi
(b))-ct(1-t)\left( \varphi (a)-\varphi (b)\right) ^{2}.
\end{equation*}%
Lemma \ref{l}, allows us to integrate both sides of this inequality on $%
[0,1] $, we get%
\begin{equation*}
\dint\limits_{0}^{1}f(t\varphi (a)+(1-t)\varphi (b))dt\leq f(\varphi
(a))\dint\limits_{0}^{1}tdt+f(\varphi
(b))\dint\limits_{0}^{1}(1-t)dt-c\left( \varphi (a)-\varphi (b)\right)
^{2}\dint\limits_{0}^{1}t(1-t)dt.
\end{equation*}%
The previous substitution in the first side of this inequality leads to%
\begin{equation*}
\frac{1}{\left( \varphi (a)-\varphi (b)\right) }\dint\limits_{\varphi
(b)}^{\varphi (a)}f\left( x\right) dx\leq \frac{f(\varphi (a))}{2}+\frac{%
f(\varphi (b))}{2}-\frac{c}{6}\left( \varphi (a)-\varphi (b)\right) ^{2}
\end{equation*}%
which gives the second inequality of (\ref{4}). This completes to proof.
\end{proof}

\begin{theorem}
If $f:[a,b]\rightarrow \mathbb{R}$ is strongly $\varphi $- convex with
modulus $c>0$ for the continuous function $\varphi :[a,b]\rightarrow \lbrack
a,b],$ then%
\begin{eqnarray}
&&\frac{1}{\varphi (b)-\varphi (a)}\dint\limits_{\varphi (a)}^{\varphi
(b)}f\left( x\right) f\left( a+b-x\right) dx  \label{5} \\
&&  \notag \\
&\leq &\frac{\left[ f^{2}(\varphi (x))+f^{2}(\varphi (y))\right] }{6}+\frac{%
2f(\varphi (x))f(\varphi (y))}{3}  \notag \\
&&  \notag \\
&&-\frac{c}{6}\left( \varphi (x)-\varphi (y)\right) ^{2}\left[ f(\varphi
(x))+f(\varphi (y))\right] -\frac{c^{2}}{30}\left( \varphi (x)-\varphi
(y)\right) ^{4}.  \notag
\end{eqnarray}
\end{theorem}

\begin{proof}
Since $f$ is strongly $\varphi $-convex with respect to $c>0$, we have that
for all $t\in \left[ 0,1\right] $%
\begin{equation}
f(t\varphi (a)+(1-t)\varphi (b))\leq tf(\varphi (a))+(1-t)f(\varphi
(b))-ct(1-t)\left( \varphi (a)-\varphi (b)\right) ^{2}  \label{6}
\end{equation}%
and 
\begin{equation}
f((1-t)\varphi (a)+t\varphi (b))\leq (1-t)f(\varphi (a))+tf(\varphi
(b))-ct(1-t)\left( \varphi (a)-\varphi (b)\right) ^{2}.  \label{7}
\end{equation}%
Multiplying both sides of (\ref{6}) by (\ref{7}), it follows that%
\begin{eqnarray}
&&f(t\varphi (a)+(1-t)\varphi (b))f((1-t)\varphi (a)+t\varphi (b))  \label{8}
\\
&&  \notag \\
&\leq &t(1-t)\left[ f^{2}(\varphi (a))+f^{2}(\varphi (b))\right] +\left(
t^{2}+(1-t)^{2}\right) f(\varphi (a))f(\varphi (b))  \notag \\
&&  \notag \\
&&-ct(1-t)\left( \varphi (a)-\varphi (b)\right) ^{2}\left[ f(\varphi
(a))+f(\varphi (b))\right] +c^{2}t^{2}(1-t)^{2}\left( \varphi (a)-\varphi
(b)\right) ^{4}.  \notag
\end{eqnarray}%
Integrating the inequality (\ref{8}) with respect to $t$ over $\left[ 0,1%
\right] $, we obtain%
\begin{eqnarray*}
&&\dint\limits_{0}^{1}f(t\varphi (a)+(1-t)\varphi (b))f((1-t)\varphi
(a)+t\varphi (b))dt \\
&& \\
&\leq &\left[ f^{2}(\varphi (a))+f^{2}(\varphi (b))\right]
\dint\limits_{0}^{1}t(1-t)dt+f(\varphi (a))f(\varphi
(b))\dint\limits_{0}^{1}\left( t^{2}+(1-t)^{2}\right) dt \\
&& \\
&&-c\left( \varphi (a)-\varphi (b)\right) ^{2}\left[ f(\varphi
(a))+f(\varphi (b))\right] \dint\limits_{0}^{1}t(1-t)dt-c^{2}\left( \varphi
(a)-\varphi (b)\right) ^{4}\dint\limits_{0}^{1}t^{2}(1-t)^{2}dt \\
&& \\
&=&\frac{\left[ f^{2}(\varphi (a))+f^{2}(\varphi (b))\right] }{6}+\frac{%
2f(\varphi (a))f(\varphi (b))}{3} \\
&& \\
&&-\frac{c}{6}\left( \varphi (a)-\varphi (b)\right) ^{2}\left[ f(\varphi
(a))+f(\varphi (b))\right] -\frac{c^{2}}{30}\left( \varphi (a)-\varphi
(b)\right) ^{4}.
\end{eqnarray*}%
If we change the variable $x:=t\varphi (a)+(1-t)\varphi (b)$, $t\in \left[
0,1\right] $, we get the required inequality in (\ref{5}). This proves the
theorem.
\end{proof}

\begin{theorem}
If $f,g:[a,b]\rightarrow \mathbb{R}$ is strongly $\varphi $- convex with
modulus $c>0$ for the continuous function $\varphi :[a,b]\rightarrow \lbrack
a,b],$ then%
\begin{eqnarray}
&&\frac{1}{\varphi (b)-\varphi (a)}\dint\limits_{\varphi (a)}^{\varphi
(b)}f(x)dx  \label{9} \\
&&  \notag \\
&\leq &\frac{M(a,b)}{3}+\frac{N(a,b)}{6}-\frac{c}{12}\left( \varphi
(a)-\varphi (b)\right) ^{2}S(a,b)+\frac{c^{2}}{30}\left( \varphi (a)-\varphi
(b)\right) ^{4}  \notag
\end{eqnarray}%
where%
\begin{equation*}
M(a,b)=f(\varphi (a))g(\varphi (a))+f(\varphi (b))g(\varphi (b))
\end{equation*}%
\begin{equation*}
N(a,b)=f(\varphi (a))g(\varphi (b))+f(\varphi (b))g(\varphi (a))
\end{equation*}%
\begin{equation*}
S(a,b)=f\left( \varphi (a)\right) +f\left( \varphi (b)\right) +g\left(
\varphi (a)\right) +g\left( \varphi (b)\right) .
\end{equation*}
\end{theorem}

\begin{proof}
Since$f,g:[a,b]\rightarrow \mathbb{R}$ is strongly $\varphi $- convex with
modulus $c>0$, we have%
\begin{equation}
f\left( t\varphi (a)+(1-t)\varphi (b)\right) \leq tf\left( \varphi
(a)\right) +(1-t)f\left( \varphi (b)\right) -ct\left( 1-t\right) \left(
\varphi (a)-\varphi (b)\right) ^{2}  \label{11}
\end{equation}%
\begin{equation}
g\left( t\varphi (a)+(1-t)\varphi (b)\right) \leq tg\left( \varphi
(a)\right) +(1-t)g\left( \varphi (b)\right) -ct\left( 1-t\right) \left(
\varphi (a)-\varphi (b)\right) ^{2}.  \label{12}
\end{equation}%
Multiplying both sides of (\ref{11}) by (\ref{12}), it follows that%
\begin{eqnarray*}
&&f\left( t\varphi (a)+(1-t)\varphi (b)\right) g\left( t\varphi
(a)+(1-t)\varphi (b)\right) \\
&& \\
&\leq &t^{2}f\left( \varphi (a)\right) g\left( \varphi (a)\right)
+(1-t)^{2}f(\varphi (b))f(\varphi (b)) \\
&& \\
&&+t(1-t)\left[ f(\varphi (a))g(\varphi (b))+f(\varphi (b))g(\varphi (a))%
\right] \\
&& \\
&&-ct^{2}\left( 1-t\right) \left( \varphi (a)-\varphi (b)\right) ^{2}\left[
f\left( \varphi (a)\right) +g\left( \varphi (a)\right) \right] \\
&& \\
&&-ct\left( 1-t\right) ^{2}\left( \varphi (a)-\varphi (b)\right) ^{2}\left[
f\left( \varphi (b)\right) +g\left( \varphi (b)\right) \right] \\
&& \\
&&+c^{2}t^{2}\left( 1-t\right) ^{2}\left( \varphi (a)-\varphi (b)\right)
^{4}.
\end{eqnarray*}%
By Lemma \ref{l}, it follows that the previous inequality can be integrated
with respect to $t$ over $[0,1]$, getting%
\begin{eqnarray*}
&&\dint\limits_{0}^{1}f\left( t\varphi (a)+(1-t)\varphi (b)\right) g\left(
t\varphi (a)+(1-t)\varphi (b)\right) dt \\
&\leq &f\left( \varphi (a)\right) g\left( \varphi (a)\right)
\dint\limits_{0}^{1}t^{2}dt+f(\varphi (b))f(\varphi
(b))\dint\limits_{0}^{1}(1-t)^{2}dt \\
&&+\left[ f(\varphi (a))g(\varphi (b))+f(\varphi (b))g(\varphi (a))\right]
\dint\limits_{0}^{1}t(1-t)dt \\
&&-c\left( \varphi (a)-\varphi (b)\right) ^{2}\left[ f\left( \varphi
(a)\right) +g\left( \varphi (a)\right) \right] \dint\limits_{0}^{1}t^{2}%
\left( 1-t\right) dt \\
&&-c\left( \varphi (a)-\varphi (b)\right) ^{2}\left[ f\left( \varphi
(b)\right) +g\left( \varphi (b)\right) \right] \dint\limits_{0}^{1}t\left(
1-t\right) ^{2}dt \\
&&+c^{2}\left( \varphi (a)-\varphi (b)\right)
^{4}\dint\limits_{0}^{1}t^{2}\left( 1-t\right) ^{2}dt.
\end{eqnarray*}%
In the first integral, we substitute $x=t\varphi (a)+(1-t)\varphi (b)$ and
simple integrals calculated, we obtain the required inequality in (\ref{9}).
\end{proof}

\end{document}